\documentclass[12pt, a4paper]{article}
\title{Necessary conditions for the existence of Morita Contexts in the bicategory of Landau-Ginzburg Models}
\author{Yves Baudelaire Fomatati\\
\small Department of Mathematics and Statistics, University of Ottawa,\\ \small Ottawa, Ontario, Canada K1N 6N5.\\ \small yfomatat@uottawa.ca.}

\date{}

\usepackage[square,comma,numbers,sort&compress]{natbib}
\usepackage[centertags]{amsmath}
\usepackage{latexsym}
\usepackage{leqno}
\usepackage{amsfonts}
\usepackage{amsmath,amssymb}
\usepackage{enumerate}

\usepackage{amsthm}
\usepackage[utf8,latin9]{inputenc}

\theoremstyle{plain}
\newtheorem{remark}{Remark}[section]
\theoremstyle{plain}
\newtheorem{lemma}{Lemma}[section]
\theoremstyle{plain}
\newtheorem{proposition}{Proposition}[section]
\theoremstyle{plain}
\newtheorem{theorem}{Theorem}[section]
\theoremstyle{plain}
\newtheorem{definition}{Definition}[section]
\theoremstyle{plain}
\newtheorem{corollary}{Corollary}[section]
\theoremstyle{plain}
\newtheorem{fact}{Fact}[section]
\theoremstyle{plain}
\newtheorem{notation}{Notations}[section]
\newtheorem{example}{Example}[section]

\theoremstyle{plain}

\usepackage[all, 2cell]{xy}
\usepackage{txfonts}

\begin{document}
\maketitle
\begin{quote}
  \textbf{Abstract}
\end{quote}

We use a matrix approach to study the concept of Morita context in the bicategory $\mathcal{LG}_{K}$ of Landau-Ginzburg models on a particular class of objects.
In fact, we first use properties of matrix factorizations to state and prove two necessary conditions to obtain a Morita context between two objects of $\mathcal{LG}_{K}$.
Next, we use a celebrated result (due to Schur) on determinants of block matrices to show that these necessary conditions are not sufficient. Finally, we state a trivial sufficient condition.
\\\\
In the sequel, $K$ is a commutative ring with unity and $R$ will denote the power series ring $K[[x_{1},\cdots,x_{n}]]$ or its subring of polynomials $K[x_{1},\cdots,x_{n}]$. It would always be clear which ring we are referring to.
\section{Introduction}

A \textit{Morita context} also called pre-equivalence data \cite{bass1967lectures}, is a generalization of \textit{Morita equivalence} between categories of modules.
Two rings $T$ and $S$ are called Morita equivalent if the categories of left $T-$Modules ($T-$Mod) and of left $S-$Modules $(S-$Mod) are equivalent.
The prototype of Morita equivalent rings is provided by a ring $T$ and the ring of $n\times n$ matrices over $T$ (for details, see corollary 22.6 of \cite{anderson2012rings}). \\
It is evident that if two rings are isomorphic, they are Morita equivalent. However, the converse is not true in general. There exist rings that are not isomorphic, yet are Morita equivalent (cf. p. 470 of \cite{lam1999graduate}). In fact, it suffices to take a ring $T$ and the ring of $n\times n$ matrices over $T$.  However, there is a partial converse which holds. Indeed,
  if two rings are Morita equivalent, then their centers are isomorphic. In particular, if the rings are commutative, then they are isomorphic (\cite{lam1999graduate}, p.494).
Because of this result, Morita equivalence is interesting solely in the situation of noncommutative rings. For more details on Morita equivalence, see \cite{anderson2012rings}. This notion can be generalized using the notion of Morita contexts.\\
Morita contexts were first introduced in the bicategory of unitary rings
and bimodules as $6$-tuples $(A,B,M,N,\phi,\psi)$; where $A$ and $B$ are rings, $M$ is an
A-B-bimodule, $N$ is a B-A-bimodule and $\phi: M\otimes N \rightarrow A$ and $\psi: N\otimes M \rightarrow B$ are homomorphisms satisfying $\phi \otimes M= M\otimes \psi$ and $N\otimes \phi = \psi \otimes N$. For a characterization of Morita context in this bicategory, see theorem 5.4 of \cite{pecsi2012morita}.
Bass (cf. chap. 2, section 4.4 of \cite{bass1962morita}) proved that the Morita context $(A,B,M,N,\phi,\psi)$ is a Morita equivalence if and only if $M$ is both projective and a generator
in the category of $A-$modules. Recall (cf. \cite{lam1999graduate}) that if $A$ is a ring, an $A$-module $P$ is projective if for every surjective $A$-linear map $f:M \rightarrow N$ and  every $A$-linear map $g:P\rightarrow N$ there is a unique $A$-linear map $h:P\rightarrow M$ such that $g=fh$.\\
 Morita contexts were recently studied in many bicategories \cite{pecsi2012morita}. But they have not yet been studied in the bicategory $\mathcal{LG}_{K}$ of Landau-Ginzburg models (section 2.2 of \cite{camacho2015matrix}) over a commutative ring $K$ whose intricate construction (\cite{carqueville2016adjunctions}) is reminiscent of, but more complex than that of the bicategory of associative algebras and bimodules. In this paper, we study the concept of Morita context in the bicategory of Landau-Ginzburg models. A Landau-Ginzburg model is a model in solid state physics for superconductivity. $\mathcal{LG}_{K}$ possesses adjoints (also called duals, cf. \cite{carqueville2016adjunctions}) and this helps in explaining a certain duality that exists in the setting of Landau-Ginzburg models in terms of some specified relations (cf. page 1 of \cite{carqueville2016adjunctions}). The objects of $\mathcal{LG}_{K}$ are \textit{potentials} which are polynomials satisfying some conditions (definition 2.4 p.8 of \cite{carqueville2016adjunctions}).
But unlike \cite{carqueville2016adjunctions}, we do not restrict ourselves to \textit{potentials}.
It turns out that the authors of \cite{carqueville2016adjunctions} used potentials to suit their purposes because even if we take the objects of $\mathcal{LG}_{K}$ to be polynomials rather than potentials and then apply the construction of $\mathcal{LG}_{K}$ given in \cite{carqueville2016adjunctions}, we obtain virtually the same bicategory except that we now have more objects. Thus in this paper, the objects of $\mathcal{LG}_{K}$ are simply polynomials.
\\
There are many reasons for studying the notion of Morita context.
The first reason is that it generalizes the very important notion of Morita equivalence.
Another reason is that it is used to prove some celebrated results. For example, the Morita context which has been introduced in \cite{morita1958duality} was used since to prove Wedderburn theorem on the structure of simple rings \cite{bass1962morita}. Morita contexts were also used in \cite{amitsur1971rings} to obtain various results: Goldie's theorem (\cite{goldie1958structure}, \cite{goldie1960semi}) on the ring of quotients of semi-prime rings and as a specialization, Wedderburn's structure theorems of semi-simple Artinian rings were obtained.
 Other applications, though sometimes not stated in an explicit form, can be found in various places (e.g. \cite{jacobson1964structure}, p.75]).\\
 In this paper, we use properties of matrix factorizations to give necessary conditions to obtain a Morita context between two objects of the bicategory $\mathcal{LG}_{K}$. Thus, our first main result is the following theorem.\\\\
\textbf{Theorem A.} \\
Let
\begin{itemize}
\item $(R,f)$ and $(S,g)$ be two objects of $\mathcal{LG}_{K}$.
\item $X\in \mathcal{LG}_{K}((R,f),(S,g))= hmf(R\otimes_{K} S, g-f)$ i.e., $X:(R,f) \rightarrow (S,g)$ is a finite rank matrix factorization of $g-f$.
\item $Y\in \mathcal{LG}_{K}((S,g),(R,f))= hmf(S\otimes_{K} R, f-g)$ i.e., $Y:(S,g) \rightarrow (R,f)$ is a finite rank matrix factorization of $f-g$. \\
    such that $X\otimes Y$ and $Y\otimes X$ are finite rank matrix factorizations.
\item $\Delta_{f}\in \mathcal{LG}_{K}((R,f),(R,f))= hmf(R\otimes_{K} R, f\otimes id - id\otimes f)$ i.e., $\Delta_{f}:(R,f) \rightarrow (R,f)$ is a finite rank matrix factorization of $f\otimes id-id\otimes f$.

\item $\Delta_{g}\in \mathcal{LG}_{K}((S,g),(S,g))= hmf(R\otimes_{K} S, g\otimes id - id\otimes g)$ i.e., $\Delta_{g}:(S,g) \rightarrow (S,g)$ is a finite rank matrix factorization of $g\otimes id-id\otimes g$.
\item $\eta: X\otimes_{R} Y \rightarrow 1_{(R,f)}=\Delta_{f}$ and let $(P,Q)$ and $(R,T)$ be pairs of matrices representing respectively the finite rank matrix factorizations $X\otimes Y$ and $\Delta_{f}$.
\item $\rho: Y\otimes_{S} X \rightarrow 1_{(S,g)}=\Delta_{g}$ and let $(P',Q')$ and $(R',T')$ be pairs of matrices representing respectively the finite rank matrix factorizations $Y\otimes X$ and $\Delta_{g}$.
 \end{itemize}
 Then:
 A necessary condition for $\Gamma = (X,Y,\eta,\rho)$ to be a \textit{Morita Context} is $$\begin{cases}
    \eta^{1}=0,\;and\; \eta^{0}Q=0,\\
    \rho^{1}=0,\;and\; \rho^{0}Q'=0.
   \end{cases} $$
   where for ease of notation, we wrote $\rho^{i}$ and $\eta^{i}$ respectively for the matrices of $\rho^{i}$ and $\eta^{i}$, $i=0,1$.
\\

Next, thanks to a celebrated result (due to Schur \cite{puntanen2005historical}, \cite{silvester2000determinants}) on determinants of block matrices, we observe that if $X$ and $X'$ are respectively two matrix factorizations of two arbitrary polynomials, then the determinants of the four matrices appearing in the Yoshino tensor products $X\widehat{\otimes} X'=(P,Q)$ and $X'\widehat{\otimes }X=(P',Q')$ are all equal.\\
Moreover, when we translate this in $\mathcal{LG}_{K}$ where a $1-$morphism is a matrix factorization of the difference of two polynomials, we find out that those four determinants are all equal to zero. So our second main result is stated as follows:\\\\
 \textbf{Theorem B.}  \\
 Let $(R,f)$ and $(S,g)$ be two objects in $\mathcal{LG}_{K}$ and let $$X:(R,f) \rightarrow (S,g)\,\, and\,\, X': (S,g)\rightarrow (R,f)$$ be 1-morphisms in $\mathcal{LG}_{K}$. If $X\otimes X'=(P,Q)$ and $X'\otimes X=(P',Q')$,
 then $$det(P)=det(Q)=det(P')=det(Q')=0$$
\\

 Thanks to this result we conclude that the necessary conditions earlier stated are not sufficient.
%
\\
 This paper is organized as follows: In the next section, we review the notion of matrix factorization. In section 3, we recall properties of matrix factorizations. Section 4 is a recall of the definition of the bicategory of Landau-Ginzburg models. The notion of Morita contexts in $\mathcal{LG}_{K}$ is discussed in section 5. Finally, we discuss further problems in the last section.
\section{Matrix Factorizations}
In this section, we first recall the definition of a matrix factorization and describe the category of matrix factorizations of a power series $f$. Next, the definition of Yoshino's tensor product is recalled.

In 1980, Eisenbud came up with an approach of factoring both reducible and irreducible polynomials in $R$ using matrices.
For instance, the polynomial
$f=x^{2}+y^{2}$ is irreducible over the real numbers but can be factorized as follows:
$$\begin{bmatrix}
    x  &  -y      \\
    y  &  x
\end{bmatrix}
\begin{bmatrix}
    x  &  y     \\
    -y  & x
\end{bmatrix}
= (x^{2} + y^{2})\begin{bmatrix}
    1  &  0      \\
    0  &  1
\end{bmatrix}
=fI_{2} $$
We say that
$
(\begin{bmatrix}
    x  &  -y      \\
    y  &  x
\end{bmatrix},
\begin{bmatrix}
    x  &  y      \\
    -y  & x
\end{bmatrix})
$
 is a $2 \times 2$ matrix factorization of $f$.

\begin{definition}\cite{yoshino1998tensor}, \cite{crisler2016matrix}  \label{defn matrix facto of polyn}   \\
An $n\times n$ \textbf{matrix factorization} of a power series $f\in \;R$ is a pair of $n$ $\times$ $n$ matrices $(P,Q)$ such that
$PQ=fI_{n}$, where $I_{n}$ is the $n \times n$ identity matrix and the coefficients of $P$ and of $Q$ are taken from $R$.
\end{definition}
When $n=1$, we get a $1$ $\times$ $1$ matrix factorization of $f$, i.e., $f=[g][h]$ which is simply a factorization of $f$ in the classical sense. But in case $f$ is not reducible, this is not interesting, that's why we will mostly consider $n > 1$.\\
  The original definition of a matrix factorization was given by Eisenbud \cite{eisenbud1980homological} as follows: a matrix factorization of an element $x$ in a ring $A$ (with unity) is an ordered pair of maps of free $A-$modules $\phi: F\rightarrow G$ and $\psi: G \rightarrow F$ s.t., $\phi\psi=x\cdot 1_{G}$ and $\psi\phi=x\cdot 1_{F}$.
  Though this definition is valid for any arbitrary ring (with unity), in order to effectively study matrix factorizations, it is important to restrict oneself to specific rings. Working with specific rings makes it possible to easily give examples and it also allows one to carry out computations in a well-defined framework. Yoshino \cite{yoshino1998tensor} restricted himself to matrix factorizations of power series. In this section, we will restrict ourselves to matrix factorizations of a  polynomial.

%
%
%
%
%
%
\begin{example} \label{third example matrix facto}
  Let $h=xy+xz^{2}+yz^{2}$.\\
We give a $2 \times 2$ matrix factorization of $h$:

$$\begin{bmatrix}
    z^{2}  &  y      \\
    x  &  -x-y
\end{bmatrix}
\begin{bmatrix}
    x+y  &  y      \\
    x  & -z^{2}
\end{bmatrix}
= (xy+xz^{2}+yz^{2})\begin{bmatrix}
    1  &  0      \\
    0  &  1
\end{bmatrix}
=hI_{2} $$
Thus;

\((\begin{bmatrix}
   z^{2}  &  y      \\
    x  &  -x-y
\end{bmatrix},
\begin{bmatrix}
     x+y  &  y      \\
    x  & -z^{2}
\end{bmatrix} )\)

 is a $2 \times 2$ matrix factorization of $h$.
\end{example}

We now propose a simple straightforward algorithm to obtain an $n\times n$ matrix factorization from one that is $m\times m$, where $m<n$.\\

\textbf{Simple straightforward algorithm}:
Let $(P,Q)$ be an $m\times m$ matrix factorizations of a power series $f$. Suppose we want an $n\times n$ matrix factorization of $f$, where $n>m$. \\
Let $(i,j)$ stand for the entry in the $i^{th}$ row and $j^{th}$ column.

\begin{itemize}
\item Turn $P$ and $Q$ into $n\times n$ matrices by filling them with zeroes everywhere except at entries $(i,j)$ where $1\leq i\leq m$ and $1\leq j\leq m$.\\
    Then for all entries $(k,k)$ with $k>m$, Either:
\begin{enumerate}
\item In $P$, replace the diagonal elements (which are zeroes) with $f$
\\ And
\item In $Q$, replace the diagonal elements (which are zeroes) with 1
\end{enumerate}
  Or:\\
  Interchange the roles of $P$ and $Q$ in steps $1.$ and $2.$ above.

\end{itemize}


It is evident that this simple algorithm works not only for polynomials but also for any element in a unital ring.\\

The standard algorithm to factor polynomials using matrices is found in \cite{crisler2016matrix} and for an improved version see \cite{fomatati2019multiplicative} and  \cite{fomatati2021tensor}.

\subsection{The category of matrix factorizations of $f\in R$} \label{subsec category of matrix facto of f}
The category of matrix factorizations of a power series $f\in R=K[[x]]:=K[[x_{1},\cdots,x_{n}]]$ denoted by $MF(R,f)$ or $MF_{R}(f)$, (or even $MF(f)$ when there is no risk of confusion) is defined \cite{yoshino1998tensor} as follows:\\
$\bullet$ The objects are the matrix factorizations of $f$.\\
$\bullet$ Given two matrix factorizations of $f$; $(\phi_{1},\psi_{1})$ and $(\phi_{2},\psi_{2})$ respectively of sizes $n_{1}$ and $n_{2}$, a morphism from $(\phi_{1},\psi_{1})$ to $(\phi_{2},\psi_{2})$ is a pair of matrices $(\alpha,\beta)$ each of size $n_{2}\times n_{1}$ which makes the following diagram commute:
$$\xymatrix@ R=0.6in @ C=.75in{K[[x]]^{n_{1}} \ar[r]^{\psi_{1}} \ar[d]_{\alpha} &
K[[x]]^{n_{1}} \ar[d]^{\beta} \ar[r]^{\phi_{1}} & K[[x]]^{n_{1}}\ar[d]^{\alpha \;\;\;\;\;\;\;\;\;\;(\bigstar)}\\
K[[x]]^{n_{2}} \ar[r]^{\psi_{2}} & K[[x]]^{n_{2}}\ar[r]^{\phi_{2}} & K[[x]]^{n_{2}}}$$
That is,
$$\begin{cases}
 \alpha\phi_{1}=\phi_{2}\beta  \\
 \psi_{2}\alpha= \beta\psi_{1}
\end{cases}$$
For a detailed definition of this category, see \cite{fomatati2019multiplicative}.\\

\subsection{Yoshino's Tensor Product of Matrix Factorizations and its variants}\label{Yoshino tensor pdt}
In this subsection, we recall the definition of the tensor product of matrix factorizations denoted $\widehat{\otimes}$, constructed by Yoshino \cite{yoshino1998tensor} using matrices. This will be useful when we will be describing the notion of \textit{Morita Context} in $\mathcal{LG}_{K}$ (cf. section \ref{morita context in LG}).
In the sequel, except otherwise stated
$S=K[[y_{1},\cdots,y_{r}]]=K[[y]]$, where $K$ is a commutative ring with unity and $r\geq 1$.
\begin{definition} \label{defn Yoshino tensor prodt} \cite{yoshino1998tensor}
Let $X=(\phi,\psi)$ be an $n\times n$ matrix factorization of a power series $f\in R$  and $X'=(\phi',\psi')$ an $m\times m$ matrix factorization of $g\in S$. These matrices can be considered as matrices over $L=K[[x,y]]$ and the \textbf{tensor product} $X\widehat{\otimes} X'$ is given by\\
(\(
\begin{bmatrix}
    \phi\otimes 1_{m}  &  1_{n}\otimes \phi'      \\
   -1_{n}\otimes \psi'  &  \psi\otimes 1_{m}
\end{bmatrix}
,
\begin{bmatrix}
    \psi\otimes 1_{m}  &  -1_{n}\otimes \phi'      \\
    1_{n}\otimes \psi'  &  \phi\otimes 1_{m}
\end{bmatrix}
\))\\
where each component is an endomorphism on $L^{n}\otimes L^{m}$.
\end{definition}
 It is easy to see \cite{yoshino1998tensor} that
 $X\widehat{\otimes} X'$ is a matrix factorization of $f+g$ of size $2nm$.\\
In the following example, we consider matrix factorizations in one variable.
\begin{example}
Let $X=(x^{2},x^{2})$ and $X'=(y^{2},y^{4})$ be $1\times 1$ matrix factorizations of $f=x^{4}$ and $g=y^{6}$ respectively. Then
\\
          \(X\widehat{\otimes} X'=(
                 \begin{bmatrix}
                  x^{2}\otimes 1 & 1\otimes y^{2}\\
                  -1\otimes y^{4} & x^{2}\otimes 1
                  \end{bmatrix}
                  ,
                  \begin{bmatrix}
                  x^{2}\otimes 1 & -1\otimes y^{2}\\
                  1\otimes y^{4} & x^{2}\otimes 1
                  \end{bmatrix}
                  )=(\begin{bmatrix}
                  x^{2} &  y^{2}\\
                  - y^{4} & x^{2}
                  \end{bmatrix}
                  ,
                  \begin{bmatrix}
                  x^{2} & -y^{2}\\
                   y^{4} & x^{2}
                  \end{bmatrix})\)
\\
And \\\\
            \(\begin{bmatrix}
                  x^{2} &  y^{2}\\
                  - y^{4} & x^{2}
                  \end{bmatrix}
                  \begin{bmatrix}
                  x^{2} & -y^{2}\\
                   y^{4} & x^{2}
                  \end{bmatrix}
                  =(x^{2}+y^{6})\begin{bmatrix}
                  1 & 0\\
                  0 & 1
                  \end{bmatrix} \)
 \\\\
 This shows that $X\widehat{\otimes} X'$ is a matrix factorization of $f+g=x^{2}+y^{6}$ and it is of size $2(1)(1)=2$.
\end{example}
In the next example, we consider matrix factorizations in two variables.
\begin{example}\label{exple mat facto 2 vars}
Let \\\\
X=
(\(
\begin{bmatrix}
    x  &  -y      \\
    y  &  x
\end{bmatrix}
,
\begin{bmatrix}
    x  &  y      \\
    -y  &  x
\end{bmatrix}
\)), and
X'=
(\(
\begin{bmatrix}
    -x  &  y      \\
    -y  &  -x
\end{bmatrix}
,
\begin{bmatrix}
    x  &  y      \\
    -y  &  x
\end{bmatrix}
\))
\\
be $2\times 2$ matrix factorizations of $f=x^{2}+y^{2}$ and $g=-x^{2}-y^{2}$ respectively. Let
\\\\
A=
\(
\begin{bmatrix}
    x  &  -y      \\
    y  &  x
\end{bmatrix}
\otimes
\begin{bmatrix}
    1  &  0      \\
    0  &  1
\end{bmatrix}
\),
B=
\(
\begin{bmatrix}
    1  &  0      \\
    0 &  1
\end{bmatrix}
\otimes
\begin{bmatrix}
    -x  &  y      \\
    -y  &  -x
\end{bmatrix}
\),
C=
\(
\begin{bmatrix}
    -1 &  0      \\
    0  &  -1
\end{bmatrix}
\otimes
\begin{bmatrix}
    x  &  y      \\
    -y  &  x
\end{bmatrix}
\), \\\\
D=
\(
\begin{bmatrix}
    x  &  y      \\
    -y  &  x
\end{bmatrix}
\otimes
\begin{bmatrix}
    1  &  0      \\
    0  &  1
\end{bmatrix}
\),
A'=
\(
\begin{bmatrix}
    x  &  y      \\
    -y  &  x
\end{bmatrix}
\otimes
\begin{bmatrix}
    1  &  0      \\
    0  &  1
\end{bmatrix}
\),
B'=
\(
\begin{bmatrix}
    -1  &  0      \\
    0 &  -1
\end{bmatrix}
\otimes
\begin{bmatrix}
    -x  &  y      \\
    -y  &  -x
\end{bmatrix}
\),\\\\
C'=
\(
\begin{bmatrix}
    1 &  0      \\
    0  &  1
\end{bmatrix}
\otimes
\begin{bmatrix}
    x  &  y      \\
    -y  &  x
\end{bmatrix}
\), and
D'=
\(
\begin{bmatrix}
    x  &  -y      \\
    y  &  x
\end{bmatrix}
\otimes
\begin{bmatrix}
    1  &  0      \\
    0  &  1
\end{bmatrix}
\)\\\\
Then
\\
          \(X\widehat{\otimes} X'=(
                 \begin{bmatrix}
                  A & B\\
                  C & D
                  \end{bmatrix}
                  ,
                  \begin{bmatrix}
                  A' & B'\\
                  C' & D'
                  \end{bmatrix}
                  )\)
\\
\\
          \[=(\left(
                  \begin{array}{cccccccc}
                    x & 0 & -y & 0 & -x & y & 0 & 0 \\
                    0 & x & 0 & -y & -y & -x & 0 & 0 \\
                    y & 0 & x & 0 & 0 & 0 & -x & y \\
                    0 & y & 0 & x & 0 & 0 & -y & -x \\
                    -x & -y & 0 & 0 & x & 0 & y & 0 \\
                    y & -x & 0 & 0 & 0 & x & 0 & y \\
                    0 & 0 & -x & -y & -y & 0 & x & 0 \\
                    0 & 0 & y & -x & 0 & -y & 0 & x \\
                  \end{array}
                \right),\left(
                  \begin{array}{cccccccc}
                    x & 0 & y & 0 & x & -y & 0 & 0 \\
                    0 & x & 0 & y & y & x & 0 & 0 \\
                    -y & 0 & x & 0 & 0 & 0 & x & -y \\
                    0 & -y & 0 & x & 0 & 0 & y & x \\
                    x & y & 0 & 0 & x & 0 & -y & 0 \\
                    -y & x & 0 & 0 & 0 & x & 0 & -y \\
                    0 & 0 & x & y & y & 0 & x & 0 \\
                    0 & 0 & -y & x & 0 & y & 0 & x \\
                  \end{array}
                \right))\]
           \\
           If we call these two $8\times 8$ matrices $P$ and $Q$ respectively, then
%
$PQ=0\cdot I_{8}=0,$ where $I_{8}$ is the identity matrix of size $8$ and the last zero is the zero matrix of size $8$.
Hence, $X\widehat{\otimes} X'$ is a matrix factorization of $f+g=x^{2}+y^{2}-x^{2}-y^{2}=0$ of size $2(2)(2)=8$.
\end{example}

Yoshino's tensor product has three mutually distinct functorial variants as can be seen in \cite{fomatati2021tensor}
\section{Properties of Matrix Factorizations}

We will now state and prove some properties of matrix factorizations of polynomials. Some of them will be used when studying the notion of Morita Context between two objects in the bicategory $\mathcal{LG}_{K}$.\\
In this section, except otherwise stated $R=K[x_{1},\cdots,x_{n}]=K[x]$.
All statements and proofs in this section are taken from \cite{crisler2016matrix} except for proposition \ref{transpose of matrix facto}. They are reproduced here
(sometimes with slight modifications) for the purposes of completeness.\\
\begin{lemma}
If $PQ=fI_{n}$, then det($P$) divides $f^{n}$. If in addition, $f$ is irreducible in $R$, then det($P$) is a power of $f$.
\end{lemma}
\begin{corollary}\label{matrix facto is invertible 1}
If $0\neq f\in R$ and $PQ=fI_{n}$, then over the field of fractions $\mathcal{F}$ of $R$, $P$ is invertible.
\end{corollary}

Since $P$ is invertible over $\mathcal{F}$, the unique $Q$ such that $PQ=fI_{n}$ is $Q=P^{-1}fI_{n}=\frac{f}{det(P)}adj(P)$, where $adj(P)$ is the adjoint of the matrix $P$. \\

Now, $adj(P)$ is a matrix over $R=K[x_{1},\cdots,x_{n}]$. So, if $det(P)$ divides $f$, then $Q$ will also be a matrix over $R$.
However, it is possible for $Q$ not to have entries in $R$, and therefore $P$ will not appear in any matrix factorization of $f$ over $R$.
 We now prove that matrices appearing in a matrix factorization commute with each other.
\begin{proposition}\label{matrix facto commute}
If \;$0\neq f\in R$ and $P,\;Q$ are $n\times n$ matrices so that $PQ=fI_{n}$, then $QP=fI_{n}$.
\end{proposition}


\begin{corollary}\label{matrix facto is invertible 2}
If $0\neq f\in R$ and $PQ=fI_{n}$, then over the field of fractions $\mathcal{F}$ of $R$, $P$ and $Q$ are invertible.
\end{corollary}


\begin{remark}
It is important to note that corollary \ref{matrix facto is invertible 1} and corollary \ref{matrix facto is invertible 2} actually say that matrices that appear in a matrix factorization of a nonzero polynomial are invertible over the field of fractions $\mathcal{F}$ of $R$.
This result will soon be very useful when describing the notion of Morita Context in $\mathcal{LG}_{K}$.
\end{remark}


We state and prove another property of matrix factorizations thanks to which we can conclude that an $n\times n$ matrix factorization of a polynomial $f$ is not unique.

\begin{proposition} \label{transpose of matrix facto}
If $0\neq f\in R$ and $P,\;Q$ are $n\times n$ matrices such that $PQ=fI_{n}$, then $Q^{t}P^{t}=fI_{n}$,
where $Q^{t}$ (respectively $P^{t}$) stands for the transpose of $Q$ (respectively $P$).
\end{proposition}
\begin{proof}
Assume $0\neq f\in R$ and $P,\;Q$ are $n\times n$ matrices such that $PQ=fI_{n}$. then:
\begin{align*}
PQ=fI_{n} &\Rightarrow (PQ)^{t}=(fI_{n})^{t}\\
&\Rightarrow  Q^{t}P^{t}=fI_{n} \;since\;(f)^{t}=f\,viewed\;as\,a\,1\times 1\,matrix
\end{align*}
\end{proof}
Before we proceed, it is worthwhile stating some well-known facts in the literature, see notes on tensor products by Conrad (for points 1. and 2. below see theorem 4.9, example 4.11 of \cite{conrad2016tensor}, point 3. simply generalizes point 2.).
\begin{lemma}\label{free mod isos}
\begin{enumerate}
  \item Let $K[x_{1},x_{2},\cdots,x_{n}]$ and $K[x'_{1},x'_{2},\cdots,x'_{m}]$ be free $K-$modules with respective bases $\{e_{i}\}^{n}_{i=1}$ and $\{e'_{j}\}^{m}_{j=1}$. Then $\{e_{i}\otimes e'_{j}\}_{i=1,\cdots n; j=1,\cdots m}$ is a basis of $K[x_{1},x_{2},\cdots,x_{n}]\otimes K[x'_{1},x'_{2},\cdots,x'_{m}]$.
  \item The $K-$modules $K[x_{1},x_{2},\cdots,x_{n}]\otimes K[x'_{1},x'_{2},\cdots,x'_{m}]$ and $K[x_{1},x_{2},\cdots,x_{n},x'_{1},x'_{2},\cdots,x'_{m}]$ are isomorphic as $K-$modules.

  \item If we let $x$ stand for $x_{1},x_{2},\cdots,x_{n}$ and $x^{(l)}$ stand for $x^{(l)}_{1},x^{(l)}_{2},\cdots, x^{(l)}_{n_{l}}$, where $n,l,n_{l}\in \mathbb{N}$, $n=n_{0}$, and $\{x^{(l)}\}$ means $x$ with $l$ primes, e.g $x^{(2)}=x'',\;x^{(0)}:=x$.\\
       Then more generally, we have:
      $$K[x,x^{(1)},\cdots, x^{(l)}]\cong \bigotimes_{p=0,\cdots,l} K[x^{(p)}].$$
\end{enumerate}
\end{lemma}
\section{The bicategory $\mathcal{LG}_{K}$ of Landau-Ginzburg models}
We quickly recall the construction of the bicategory $\mathcal{LG}_{K}$ of Landau-Ginzburg models. In definition 5.5 of \cite{fomatati2019multiplicative},
a $B-$category is defined to be a structure having all requirements of a bicategory but without necessarily satisfying the unit morphisms requirement. In order to easily recall the construction of $\mathcal{LG}_{K}$,
we first construct a $B-$category which we call $B-Fac$ and then we proceed to the unit construction. The objects of $B-Fac$ are polynomials $f$ denoted by pairs $(R,f)$ where $f\in R=K[x]$. Here we do not impose restrictions on the objects of our bicategory as was originally done in \cite{carqueville2016adjunctions}, where the authors instead consider \textit{potentials} (Definition 2.4, p.8 of \cite{carqueville2016adjunctions}). This generalization we do at the level of the objects does not pose a problem in the construction of $\mathcal{LG}_{K}$. The end result is simply a bicategory which has more objects than the original $\mathcal{LG}_{K}$ defined in \cite{carqueville2016adjunctions}, we still call it $\mathcal{LG}_{K}$.\\
We first recall the notion of linear factorizations which is an ingredient for the construction of $\mathcal{LG}_{K}$.
\begin{definition}(p.8 of \cite{carqueville2016adjunctions}) Linear factorization \label{def lin facto}\\
A \textbf{linear factorization} of $f\in R=K[x]$ is a $\mathbb{Z}_{2}-$graded $R-$module $X=X^{0}\oplus X^{1}$
together with an odd (i.e., grade reversing) $R-$linear endomorphism $d: X\longrightarrow X$
such that $d^{2}=f\cdot id_{X}$.\\
$f\cdot id_{X}$ stands for the endomorphism $x\mapsto f\cdot x,\,\,\forall x\in X$.
\end{definition}
 Since we are dealing with a $\mathbb{Z}_{2}-$grading and $d$ is odd, we can also say $d$ is a degree one map.
$d$ is called a \textit{twisted differential} in \cite{carqueville2015toolkit}.
 $d$ is actually a pair of maps $(d^{0},d^{1})$ that we may depict as follows:
$$\xymatrix {X^{0}\ar [r]^{d^{0}} &X^{1}\ar [r]^{d^{1}} & X^{0}}$$
 and the stated condition on them is:
$d^{0}\circ d^{1}=f\cdot id_{X^{1}}$ and $d^{1}\circ d^{0}= f\cdot id_{X^{0}}$. \\ If $X$ is a free $R-$module, then the pair $(X,d)$ is called a \textit{matrix factorization} and we often refer to it by $X$ without explicitly mentioning the differential $d$.
\begin{remark}
 If $M_{0}$ and $M_{1}$ are respectively matrices of the $R-$linear endomorphisms $d^{0}$ and $d^{1}$, then the pair $(M_{0},M_{1})$ would be a matrix factorization of $f$ according to definition \ref{defn matrix facto of polyn}.
\end{remark}
\begin{definition} (p.9 \cite{carqueville2016adjunctions}) Morphism of linear factorizations \label{morphism of matrix facto}\\
  A \textbf{morphism of linear factorizations} $(X,d_{X})$ and $(Y,d_{Y})$ is an even (i.e., a grade preserving) $R-$linear map $\phi: X \longrightarrow Y$ such that $d_{Y}\phi = \phi d_{X}$.
\end{definition}
Concretely (see page 19 of \cite{khovanov2008matrix}), $\phi$ is a pair of maps $\xymatrix {X^{0}\ar [r]^{\phi^{0}} &Y^{0}}$ and $\xymatrix {X^{1}\ar [r]^{\phi^{1}} &Y^{1}}$ such that the following diagram commutes:\\

 $$\xymatrix@ R=0.6in @ C=.75in{X^{0} \ar[r]^{{d^{0}_{X}}} \ar[d]_{\phi^{0}} &
X^{1} \ar[d]^{\phi^{1}} \ar[r]^{{d^{1}_{X}}} & X^{0}\ar[d]^{\phi^{0}}\\
Y^{0} \ar[r]^{{d^{0}_{Y}}} & Y^{1}\ar[r]^{{d^{1}_{Y}}} & Y^{0}}$$

\begin{definition} \label{homotopic lin facto} (p.9 \cite{carqueville2016adjunctions}) homotopic linear factorizations\\
 Let $(X,d_{X})$ and $(Y,d_{Y})$ be linear factorizations. Two morphisms $\varphi, \psi : X\longrightarrow Y$ are \textbf{homotopic} if there exists an odd $R-$linear map $\lambda: X\longrightarrow Y$ such that $d_{Y}\lambda + \lambda d_{X}= \psi - \varphi$.\\
 More precisely, the following diagram commutes:
$$\xymatrix@ R=0.6in @ C=.75in{X_{1} \ar[r]^{d_{X}} \ar[d]_{\psi-\phi} &
X_{0}\ar[dl]_{\lambda_{0}} \ar[d]^{\psi-\phi} \ar[r]^{d_{X}} & X_{1}\ar[dl]_{\lambda_{1}}\ar[d]^{\psi-\phi \,\,\,\,\,\,\,\,\,\,\,\,\,\,\,\dag}\\
Y_{1} \ar[r]^{d_{Y}} & Y_{0}\ar[r]^{d_{Y}} & Y_{1}}$$ i.e., $$d_{Y}\circ \lambda_{0} + \lambda_{1}\circ d_{X}=\psi-\phi$$
\end{definition}

\begin{notation}\label{notations for categ of facto} We keep the following notations used in \cite{carqueville2016adjunctions}:\\
$\bullet$ We denote by $HF(R, f)$ the category of linear factorizations of $f\in R$ modulo homotopy. \\
$\bullet$ We also denote by $HMF(R, f)$ its full subcategory of matrix factorizations. \\
$\bullet$ Furthermore, we write $hmf(R, f)$ for the full subcategory of finite rank matrix factorizations, viz. the matrix factorizations whose underlying $R$-module is free of finite rank.\\
\end{notation}

\begin{remark} (p.9 of \cite{carqueville2016adjunctions})\\
$HMF(R,f)$ is idempotent complete (\cite{bokstedt1993homotopy}, \cite{neeman2001triangulated}).
As earlier stated, we work with polynomials rather than power series, so $hmf(R,f)$ is not necessarily idempotent complete \cite{keller2011two}. The idempotent closure of $hmf(R,f)$ (denoted by $hmf(R,f)^{\omega}$) is a full subcategory of $HMF(R,f)$ whose objects are those matrix factorizations which are direct summands of finite-rank matrix factorizations in the homotopy category. Moreover, $hmf(R,f)^{\omega}$ is an idempotent complete category.
 Taking the idempotent completion is necessary because the composition of $1$-morphisms in $\mathcal{LG}_{K}$ results in matrix factorisations which, while not finite-rank, are summands in the homotopy category of something finite-rank. There are two natural ways to resolve this: work throughout with power series rings and completed tensor products, or work
with idempotent completions.
\end{remark}

Let ($R=K[x], f$) and ($S=K[y], g$) be elements of $B-Fac$. The small category $B-Fac((R,f),(S,g))$ is defined as follows:
$$B-Fac((R,f),(S,g)):=hmf(R\otimes S, 1_{R} \otimes g - f\otimes 1_{S} )^{\omega}=hmf(K[x,y] , g - f )^{\omega}$$
viz. a 1-morphism between two polynomials $f$ and $g$ is a matrix factorization of $g-f$.\\
Then given two composable $1-$cells $X\in B-Fac((R,f),(S,g))$ and $Y\in B-Fac((S,g),(T,h))$, we define their composition using Yoshino's tensor product as discussed in subsection \ref{Yoshino tensor pdt}. $Y\circ X:=Y\otimes_{S} X$ $\in\,HMF(R\otimes_{K}T, 1_{R}\otimes h - f\otimes 1_{T})$ which is a $\mathbb{Z}_{2}-$graded module, where:
$$(Y\otimes_{S} X)^{0}=(Y^{0}\otimes_{S} X^{0})\oplus (Y^{1}\otimes_{S} X^{1})\;\;and \;\; (Y\otimes_{S} X)^{1}=(Y^{0}\otimes_{S} X^{1})\oplus (Y^{1}\otimes_{S} X^{0}),$$
the differential (\cite{carqueville2016adjunctions}) is
 $$d_{Y\otimes X}=d_{Y}\otimes 1 + 1\otimes d_{X},$$ where the tensor product is taken over $S$
and $1\otimes d_{X}$ has the usual Koszul signs when applied to elements. That is;
$$d_{Y\otimes X}(y,x)=d_{Y}(y)\otimes x + (-1)^{i}y\otimes d_{X}(x)\,\,\,\,\,\,\,\,\,$$ (see \, p.28\,\cite{khovanov2008matrix}) where $y\in Y^{i}$.\\
 By remark 2.1.8 on p.29 of \cite{camacho2015matrix}, $Y\otimes_{S} X$ is a free module of infinite rank over $R\otimes_{K} T$.
 However, the argument of Section 12 of \cite{dyckerhoff2013pushing} shows that it is naturally isomorphic to a direct summand in the homotopy category of something finite-rank. So, we may define $Y\circ X:=Y\otimes_{S} X$ $\in\,hmf(R\otimes_{K}T, 1_{R}\otimes h - f\otimes 1_{T})^{\omega}= B-Fac((R,f),(T,h))$.\\

We now define the tensor product of morphisms of matrix factorizations. Let $X_{1}$, $X_{2}$ be
objects of $B-Fac((R,f),(S,g))$ and $Y_{1}$, $Y_{2}$ be
objects of $B-Fac((S,g), (T,h))$. Let $\alpha : X_{1} \rightarrow X_{2}$ and $\beta : Y_{1} \rightarrow Y_{2}$ be two morphisms, then we define their tensor product in the obvious way $\beta \otimes \alpha : Y_{1} \otimes X_{1} \rightarrow Y_{2}\otimes X_{2}$ in $B-Fac((R,f),(T,h))$.\\
With the above data, the composition (bi-)functor is entirely determined in our B-category:

$\star_{(R,f),(S,g),(T,h)}:B-Fac((R,f),(S,g))$ x  $B-Fac((S,g), (T,h))\rightarrow B-Fac((R,f),(T,h))$ $$(X, Y)\mapsto Y\otimes_{S} X$$
The definition of the associativity morphism is easy to state. In fact, for $X\in B-Fac((R,f),(S,g))$, $Y\in B-Fac((S,g),(T,h))$ and $Z\in B-Fac((T,h), (P,r))$, the associator is the 2-isomorphism $$a_{Z,Y,X}: Z\otimes (Y\otimes X)\rightarrow (Z\otimes Y)\otimes X$$ given by the usual formula $$z\otimes (y\otimes x)\rightarrow (z\otimes y)\otimes x$$
where $x\in X$, $y\in Y$ and $z\in Z$.
\begin{lemma}
  $B-Fac$ is a $B-$category.
\end{lemma}
We now discuss the construction of the units in $\mathcal{LG}_{K}$.
\subsection{Unit $1-$morphisms in $\mathcal{LG}_{K}$} \label{subsec: unit 1-morphisms in LGk}
Here, we will construct the identity 1-cells. We let $R=K[x_{1},x_{2},\cdots,x_{n}]$. \\

From lemma \ref{free mod isos}, we have in particular that $$R\otimes_{K} R \cong K[x_{1},x_{2},\cdots,x_{n},x'_{1},x'_{2},\cdots,x'_{n}].$$
where $x_{i}=x_{i}\otimes 1$ and $x^{'}_{i}= 1 \otimes x_{i}$.
\\
The subscript $"K"$ in $\otimes_{K}$ will be very often omitted for ease of notation.
We need an object $\Delta_{f}: (R,f) \rightarrow (R,f)$ in
$hmf(R\otimes R,  f\otimes id - id\otimes f)$ or equivalently,\\
$hmf(K[x_{1},x_{2},\cdots,x_{n},x'_{1},x'_{2},\cdots,x'_{n}], h(x,x')),$ where $h(x,x')= f(x)- f(x'),$ where $id$ stands for $1_{R}$. \\
Recall (cf. section 5.5 of \cite{smith2011introduction}): The exterior algebra $\bigwedge(V)$ of a vector space $V$ over a field K is defined as the quotient algebra of the tensor algebra; $T(V)=\bigoplus_{i=1}^{\infty}T^{i}(V)=K\bigoplus V \bigoplus (V\otimes V)\bigoplus (V\otimes V\otimes V)\bigoplus \cdots ,$ by the two-sided ideal $I$ generated by all elements of the form $x \otimes x$ for $x \in V$. Symbolically, $\bigwedge(V)=T(V)/I$. The exterior product $\wedge$ of two elements of $\bigwedge(V)$ is the product induced by the tensor product $\otimes$ of $T(V)$. That is, if

     $\pi: T ( V )\rightarrow \bigwedge( V ) = T ( V ) / I $

is the canonical surjection, and if a and b are in $\bigwedge(V)$, then there are $\alpha$  and  $\beta$ in $T(V)$ such that $a = \pi ( \alpha )$ and $b = \pi ( \beta )$ and $a\wedge b=\pi(\alpha \otimes \beta)$.
 Let $\theta_{1},\theta_{2},\cdots, \theta_{n}$ be formal symbols\footnote{That is, we declare those symbols to be linearly independent by definition.} We consider the $R\otimes R-$module:
 $$\Delta_{f}=\bigwedge\{\bigoplus_{i=1}^{n}(R\otimes R)\theta_{i}\}$$

 This is an exterior algebra generated by $n$ anti-commuting variables, the $\theta_{i}s$
 modulo the relations that the $\theta$'s anti-commute, that is $\theta_{i}\wedge \theta_{j}=-\theta_{j}\wedge \theta_{i}$. Typically, we will omit the wedge product and write for instance $\theta_{i}\wedge \theta_{j}$ simply as $\theta_{i}\theta_{j}$. Here, the "$\wedge$" product is taken over K just like the tensor product.
 A typical element of $\Delta_{f}$ is $(r\otimes r')\theta_{i_{1}}\theta_{i_{2}}\cdots \theta_{i_{k}}$ or equivalently $h(x_{1},x_{2},\cdots,x_{n},x'_{1},x'_{2},\cdots,x'_{n})\theta_{i_{1}}\theta_{i_{2}}\cdots \theta_{i_{k}}$ where $i_{1},\cdots, i_{k}\in \{1,\cdots,n\}$ and $h(x_{1},x_{2},\cdots,x_{n},x'_{1},x'_{2},\cdots,x'_{n})$ $\in$
 $K[x_{1},x_{2},\cdots,x_{n},x'_{1},x'_{2},\cdots,x'_{n}]$.\\
 $\Delta_{f}$ as an algebra is finitely generated by the set of formal symbols $\{\theta_{1},\cdots, \theta_{n}\}$. \\
$\Delta_{f}$ as an $R\otimes R-$module is generated by the set containing the empty list and all products of the form $\theta_{i_{1}}\dots \theta_{i_{k}}$ where $i_{1},\cdots, i_{k} \in \{1,\cdots,n\}$.
  The action of $R\otimes R$ is the obvious one.

  $\Delta_{f}$ is endowed with the $\mathbb{Z}_{2}-$grading given by $\theta-$degree (where deg$\theta_{i}=1$ for each $i$). Thus $deg\theta_{i}^{2}=0\;and\;deg\theta_{i}\theta_{j}=0$.
 \\
  Next, we define the differential as follows:\\
  $$d: \Delta_{f}\longrightarrow \Delta_{f}$$
  $$d(-)=\sum_{i=1}^{n}[(x_{i}-x_{i}^{'})\theta_{i}^{\ast}(-)+\partial_{i}(f)\theta_{i}\wedge (-)]\,\,\,\,\,\,\,\,\,\,\,\,\,\,\,\,\cdots\,\,\,\natural$$
Where $\theta_{i}^{\ast}$ is the unique derivation extending the map $\theta_{i}^{\ast}(\theta_{j})=\delta_{ij}$ and as mentioned in \cite{carqueville2016adjunctions}, it acts on an element $\theta_{i_{1}}\theta_{i_{2}}\cdots \theta_{i_{k}}$ of the
exterior algebra by the Leibniz rule with Koszul signs.
In fact,
$$
\theta_{i}^{\ast}(\theta_{j_{1}}\theta_{j_{2}}\cdots \theta_{j_{k}})=
\begin{cases}
0\,for\,i\neq j,\,\forall j\in \{{j_{1}},{j_{2}},\cdots,{j_{k}}\}\\
(-1)^{p+1}\theta_{j_{1}}\theta_{j_{2}}\cdots \hat{\theta_{i}}\cdots \theta_{j_{k}}\,\,otherwise
\end{cases}
$$

where
 $\hat{\theta_{i}}$ signifies
that $\theta_{i}$ has been removed, and $p$ is the position of $\theta_{i}$ in $\theta_{j_{1}}\theta_{j_{2}}\cdots \theta_{i}\cdots \theta_{j_{k}}$\\

And:
$$\partial_{i}: k[x_{1},\cdots,x_{n},x_{1}^{'},\cdots,x_{n}^{'}]\longrightarrow k[x_{1},\cdots,x_{n},x_{1}^{'},\cdots,x_{n}^{'}]$$ is defined by,

$$\partial_{i}(h)=\frac{h(x_{1}^{'},\cdots,x_{i-1}^{'},x_{i},\cdots,x_{n},x_{1}^{'},\cdots,x_{n}^{'})-h(x_{1}^{'},
\cdots,x_{i}^{'},x_{i+1},\cdots,x_{n},x_{1}^{'},\cdots,x_{n}^{'})}{x_{i}-x_{i}^{'}}$$
where for ease of notation we wrote $h$ as argument of $\partial_{i}$ instead of the more cumbersome notation
$\partial_{i}(h(x_{1},\cdots,x_{n},x_{1}^{'},\cdots,x_{n}^{'}))$.
The following lemma will be useful in the definition of the left and the right units of $\mathcal{LG}_{K}$.
\begin{lemma} \label{projection map}

There is a canonical map of factorizations \\$\pi: \Delta_{f}\longrightarrow R$ given by
$\pi[(r\otimes r')\theta_{i_{1}}\theta_{i_{2}}\cdots \theta_{i_{k}}]=\delta_{k,0}rr'$.
$\pi$ is in fact the composition of the projection $\pi^{\ast}: \Delta_{f}\longrightarrow R\otimes R$ to $\theta-$degree $0$, followed by the multiplication map $m: R\otimes R \longrightarrow R=R_{0}\oplus R_{1}$, where we endow $R$ with the trivial grading i.e., $R_{0}=R$ and $R_{1}=\{0\}$.
\end{lemma}

\subsection{The left and the right units of $\mathcal{LG}_{K}$}
In this subsection, we recall the definition of the left and right identities of the bicategory $\mathcal{LG}_{K}$.
We will denote the right (respectively left) unit by $\rho$ (respectively $\lambda$).
\\
Consider a $1$-morphism $X\in$ $hmf(R\otimes S, 1_{R}\otimes g - f\otimes 1_{S})^{\omega}=hmf(R\otimes S, id\otimes g - f\otimes id)^{\omega}$.\\
Thus, $X$ is a matrix factorization of $id\otimes g - f\otimes id$ and is also an $R\otimes S-$module.\\
Let $1_{X}:X \longrightarrow X$ be the identity map and $\pi$ be the projection defined in lemma \ref{projection map}.
\begin{remark}
  Observe that any $S-$module $N$ can be considered as an $R-$module by letting $rn:=f(r)n$ where $f: R\longrightarrow S$ is a homomorphism of commutative rings.
It is easy to see that the $R\otimes S-$module $X$ can be considered as an $R-$module via the following $K-$homomorphism of commutative unitary rings $f: R\longrightarrow R\otimes S$ defined
 by $f(r)=r\otimes 1_{S}$,
and hence one can also see $X$ as an $R\otimes R-$module by means of the following (multiplicative map which is a) $K-$homomorphism of commutative unitary rings $m: R\otimes R\longrightarrow R$ defined by $m(r\otimes r')=rr'$.
It is not difficult to see that the $R\otimes R-$module $\Delta_{f}$ can be considered as an $R-$module via the following homomorphism of commutative unitary rings $f: R\longrightarrow R\otimes R$ defined
by $f(r)=r\otimes 1_{R}$.\\
\end{remark}
Thanks to this remark, it makes sense to form the following tensor product over $R$: $X\otimes_{R} R$ and
$X\otimes_{R} \Delta_{f}$ since $X$ and $\Delta_{f}$ can be viewed as $R-$modules. Consequently, we will simply write $X\otimes R$ and $X\otimes \Delta_{f}$ for ease of notation.\\
Similarly, since the $R\otimes S-$module $X$ and the $S\otimes S-$module $\Delta_{g}$ can be viewed as $S-$modules, we can form the module $\Delta_{g} \otimes_{S} X$ that we simply write as $\Delta_{g} \otimes X$.
\\
Also consider the map $u: X\otimes R\longrightarrow X$ defined by $u(x\otimes r)=xr$. This definition makes sense since $X$ can be viewed as an $R-$module. $u$ is an isomorphism (See example 1 page 363 of \cite{dummit2004abstract}).\\

Now, define $\rho_{X}: X\otimes \Delta_{f}\longrightarrow X$ by $\rho_{X}:=u\circ (1_{X}\otimes \pi)$ and $\lambda_{X}: \Delta_{g}\otimes X\longrightarrow X$ by $\lambda_{X}:=u\circ (\pi\otimes 1_{X})$.
\\
$\rho_{X}$ and $\lambda_{X}$ are clearly morphisms in $hmf(R\otimes S, id\otimes g - f\otimes id)^{\omega}$.
\\
$\rho$ is natural w.r.t. $2-$morphisms in the variable $X$ and there is no direct inverse
 for $\rho_{X}$, for each $X$ (See section 5.2 of \cite{fomatati2019multiplicative}).
\section{Morita contexts in $\mathcal{LG}_{K}$}\label{morita context in LG}
Before discussing the notion of Morita context in $\mathcal{LG}_{K}$, we define what it is in an arbitrary bicategory.
\begin{definition} \cite{pecsi2012morita} \label{defn morita context}\\
Let $\mathcal{B}$ be a bicategory with natural isomorphisms a, r and l. Given two 0-cells A and B, we define a \textbf{Morita Context} between A and B as a four-tuple $\Gamma = (f,g,\eta,\rho)$
consisting of two 1-cells $f\in Hom_{\mathcal{B}}(A,B)$ and $g\in Hom_{\mathcal{B}}(B,A)$, and two 2-cells $\eta: fg\rightarrow 1_{A}$ and $\rho: gf\rightarrow 1_{B}$ such that the following diagrams commute\\

$\xymatrix @ R=0.5in @ C=.6in{
(gf)g \ar[rr]^{a}\ar[d]_{\rho 1} &&
g(fg) \ar[d]^{1\eta} &\\
1_{B}g \ar[rd]_l &&
g1_{A} \ar[ld]^{r} &\\
&g} \xymatrix @ R=0.5in @ C=.6in{
(fg)f \ar[rr]^{a}\ar[d]_{\eta 1} &&
f(gf) \ar[d]^{1\rho} &\\
1_{A}f \ar[rd]_l &&
f1_{B} \ar[ld]^{r} &\\
& f} $\\

Equationally, we have:
\begin{enumerate}
\item $r\circ 1\eta \circ a= l\circ \rho 1$
\item $r\circ 1\rho \circ a= l\circ \eta 1$
\end{enumerate}
\end{definition}
\begin{remark}\cite{pecsi2012morita}
\begin{itemize}
\item Observe that any adjunction $<\xymatrix{ A \ar[r]^-f & B },\xymatrix{ B \ar[r]^-g & A },\xymatrix{ 1_{A} \ar[r]^-\varepsilon & fg },\xymatrix{ gf \ar[r]^-\eta & 1_{B} }>$
     becomes a \textit{Morita context} as soon as its unit $\varepsilon$ is invertible.
\item
A \textit{Morita context} is \textbf{strict} if both $\eta$ and $\rho$ in the foregoing definition are isomorphisms. \textit{Strict Morita contexts} and adjoint equivalences are basically the same. One can switch between them by inverting the unit $\varepsilon$ of the adjunction.
\end{itemize}
\end{remark}
\textbf{Nota Bene}: It is perhaps good to mention that what is called \textit{Morita context} in this paper is instead called \textit{wide right Morita context} from $B$ to $A$ in \cite{kaoutit2006wide}. In \cite{pecsi2012morita}, it is also called an \textit{abstract bridge}. Both authors declare that the notion of \textit{left Morita context} is defined by reversing 2-cells. We will not deal with \textit{left Morita context} in our work.\\

In the sequel, we will write $K[x]$ (respectively $K[y]$) for $K[x_{1},x_{2},\cdots,x_{n}]$ ($K[y_{1},y_{2},\cdots,y_{m}]$), where $x_{i}$ and $y_{j}$ are indeterminates for $i,j \in\{1,2,\cdots, max(n,m)\}$. \\

\textbf{Description of Morita context in $\mathcal{LG}_{K}$}\\
Let $(R,f)$ and $(S,g)$ be two objects of $\mathcal{LG}_{K}$, that is polynomials such that $f\in R= K[x]$ and $g\in S= K[y]$. In all of this section, except otherwise stated, we want to keep the following remark and assumption in mind:
\begin{remark} Let $X\in \mathcal{LG}_{K}((R,f),(S,g))$ and $Y\in \mathcal{LG}_{K}((S,g),(R,f))$ be $1-$morphisms of $\mathcal{LG}_{K}$, we normally have:
\begin{enumerate}
\item $X\in \mathcal{LG}_{K}((R,f),(S,g))= hmf(R\otimes_{K} S, g-f)^{\omega}$ i.e., $X:(R,f) \rightarrow (S,g)$ is a matrix factorization which is a direct summand of a finite rank matrix factorization of $g-f$.
\item $Y\in \mathcal{LG}_{K}((S,g),(R,f))= hmf(S\otimes_{K} R, f-g)^{\omega}$ i.e., $Y:(S,g) \rightarrow (R,f)$ is a matrix factorization which is a direct summand of a finite rank matrix factorization of  $f-g$.
\item $\Delta_{f}\in \mathcal{LG}_{K}((R,f),(R,f))= hmf(R\otimes_{K} R, f\otimes id - id\otimes f)$ i.e., $\Delta_{f}:(R,f) \rightarrow (R,f)$ is a finite rank matrix factorization of $f\otimes id - id\otimes f $.

\item $\Delta_{g}\in \mathcal{LG}_{K}((S,g),(S,g))= hmf(S\otimes_{K} S, g\otimes id - id\otimes g)$ i.e., $\Delta_{g}:(S,g) \rightarrow (S,g)$ is a finite rank matrix factorization of $g\otimes id - id\otimes g $.
\end{enumerate}
Observe that the $1-$morphism $X$ as defined in the previous remark can also be a finite rank matrix factorization i.e., an object of $hmf(R\otimes_{K} S, g-f)$ which is a subcategory of $hmf(R\otimes_{K} S, g-f)^{\omega}$. A similar observation holds for $Y$.\\
We would like to use determinant of matrices to discuss the notion of Morita context in $\mathcal{LG}_{K}$, therefore, it is important for us to deal with entities that are of finite rank. That is why we need the following assumption.\\
\textbf{\textit{Assumption}}:
We restrict our study of Morita context in $\mathcal{LG}_{K}$ to those objects $(R,f)$ and $(S,g)$ which are such that the following $1-$morphisms are all finite rank matrix factorizations: $X$, $Y$, $X\otimes Y$ and $Y\otimes X$.\\
Since we will only be dealing with finite rank matrix factorizations, we will sometimes intentionally omit the phrase "finite rank" in front of the phrase "matrix factorization",.
\end{remark}
A Morita Context between objects $(R,f)$ and $(S,g)$ of $\mathcal{LG}_{K}$ is a four-tuple
$\Gamma = (X,Y,\eta,\rho)$ where:
\begin{itemize}
\item $X:(R,f) \rightarrow (S,g)$ is a matrix factorization of $g-f$.
\item $Y:(S,g) \rightarrow (R,f)$ is a matrix factorization of $f-g$.
\item $\eta: X\otimes_{R} Y \rightarrow 1_{(R,f)}=\Delta_{f}$ where
 $\Delta_{f}$ is the identity on $(R,f)$ in $\mathcal{LG}_{K}$ i.e., a matrix factorization of $f\otimes id - id\otimes f $ as seen in the previous section.
\item $\rho: Y\otimes_{S} X \rightarrow 1_{(S,g)}=\Delta_{g}$\\

such that the following diagrams (thereafter referred to as the $M.C.LG_{K}$ diagrams) commute up to homotopy:\\

$\xymatrix @ R=0.4in @ C=.21in{
(Y\otimes_{S} X)\otimes_{R} Y \ar[rr]^{a}\ar[d]_{\rho \otimes 1_{Y}} &&
Y\otimes_{S} (X\otimes_{R} Y) \ar[d]^{1_{Y}\otimes\eta} &\\
\Delta_{g}\otimes_{S} Y \ar[rd]_{l_{Y}} &&
Y\otimes_{R} \Delta_{f} \ar[ld]^{r_{Y}} &\\
& Y} \xymatrix @ R=0.4in @ C=.21in{
(X\otimes_{R} Y)\otimes_{S} X \ar[rr]^{a}\ar[d]_{\eta \otimes 1_{X}} &&
X\otimes_{R} (Y\otimes_{S} X) \ar[d]^{1_{X}\otimes\rho} &\\
\Delta_{f}\otimes_{R} X \ar[rd]_{l_{X}} &&
X\otimes_{S} \Delta_{g} \ar[ld]^{r_{X}} &\\
& X} $\\

 That is, the following two conditions hold:
\begin{enumerate}
\item $r_{Y}\circ 1_{Y}\otimes\eta \circ a\approx l_{Y}\circ \rho\otimes 1_{Y}$
\item $r_{X}\circ 1_{X}\otimes\rho \circ a \approx l_{X}\circ \eta \otimes 1_{X}$
\end{enumerate}
where $\approx$ stands for the homotopy relation.
\end{itemize}
equivalently:
\begin{enumerate}
\item $\exists \lambda: Z \rightarrow Y$ s.t. $d_{Y}\lambda + \lambda d_{Z}= \psi - \phi$, where $Z=(Y\otimes_{S} X)\otimes_{R} Y$, $\psi=r_{Y}\circ 1_{Y}\otimes\eta \circ a$, and $\phi=l_{Y}\circ \rho\otimes 1_{Y}$

\item $\exists \xi: Z' \rightarrow X$ s.t. $d_{X}\xi + \xi d_{Z'}= \psi' - \phi'$, where $Z'=(X\otimes_{R} Y)\otimes_{S} X$, $\psi'=r_{X}\circ 1_{X}\otimes\rho \circ a$, and $\phi'=l_{X}\circ \eta \otimes 1_{X}$
\end{enumerate}
%
We would now like to find necessary conditions on $\eta$ and $\rho$
to obtain a \textit{Morita context} in $\mathcal{LG}_{K}$. We will use a matrix approach because it is easier to use matrices than linear transformations in our setting. Thus, we will have recourse to one of the properties of matrix factorizations we studied earlier (cf. corollaries \ref{matrix facto is invertible 1} and \ref{matrix facto is invertible 2}); namely that matrices that appear in a matrix factorization of a nonzero polynomial are invertible.\\
In the sequel, except otherwise stated, matrices that appear in a matrix factorization are of a fixed size
$n\in \mathbb{N}$. So we will not bother to mention the sizes of pair of matrices constituting a matrix factorization.\\
Let $(P,Q)$ and $(R,T)$ be pairs of matrices representing respectively the matrix factorizations $X\otimes_{R} Y$ and $\Delta_{f}$. We assume $f$ is not a constant polynomial, thus $\Delta_{f}$ is not the zero polynomial and so, by corollaries \ref{matrix facto is invertible 1} and \ref{matrix facto is invertible 2}, we have that $R$ (respectively $T$) is invertible over the field of fractions of $R$ (respectively over the field of fractions of $T$). \\
We know that \\
$\eta =(\eta^{0}, \eta^{1}): X\otimes Y \rightarrow 1_{(R,f)}=\Delta_{f}$ \\
is a morphism of matrix factorizations if the following diagram commutes:\\
$\xymatrix{(X\otimes Y)^{0} \ar[r]^{{d^{0}_{X\otimes Y}}} \ar[d]_{\eta^{0}} &
(X\otimes Y)^{1} \ar[d]^{\eta^{1}} \ar[r]^{{d^{1}_{X\otimes Y}}} & (X\otimes Y)^{0}\ar[d]^{\eta^{0}}\\
(\Delta_{f})^{0} \ar[r]^{{d^{0}_{\Delta_{f}}}} & (\Delta_{f})^{1}\ar[r]^{{d^{1}_{\Delta_{f}}}} & (\Delta_{f})^{0}}$.\\

That is: $$\begin{cases}
\eta^{0}d^{1}_{X\otimes Y} = d^{1}_{\Delta_{f}}\eta^{1}\\
\eta^{1}d^{0}_{X\otimes Y} = d^{0}_{\Delta_{f}}\eta^{0}
\end{cases}$$


In matrix form: $$\sharp\begin{cases}
\eta^{0}Q = T\eta^{1}...i)\\
\eta^{1}P = R\eta^{0}...ii)
\end{cases}$$
where for ease of notation we wrote $\eta^{j}$ for the matrix of $\eta^{j}$, $j=0,1$.\\
Now, $X$ (respectively $Y$) being a matrix factorization of $g(y)-f(x)$ (respectively $f(x)-g(y)$) we obtain from the Yoshino's tensor product of matrix factorizations (cf. definition 1.2 of \cite{yoshino1998tensor}) that, $X\otimes Y$ is a factorization of $(g(y)-f(x))+(f(x)-g(y))=0$. This means that $PQ=0$.

from $ii)$, since $R$ is invertible, we have:
\begin{align*}
\eta^{1}P = R\eta^{0} &\Rightarrow R^{-1}\eta^{1}P =\eta^{0}  \cdots \bigstar \\
&\Rightarrow R^{-1}\eta^{1}PQ =\eta^{0}Q\\
&\Rightarrow \eta^{0}Q = 0 \cdots \bigstar \bigstar
\end{align*}

Putting this in $i)$, since $T$ is invertible, we obtain:
\begin{align*}
T\eta^{1}= \eta^{0}Q = 0 &\Rightarrow \eta^{1}= T^{-1}0= 0\\
&\Rightarrow \eta^{1}= 0
\end{align*}

We would now like to give a necessary condition on $\rho$. The process to obtain this necessary condition is completely analogous to what we did in the case of $\eta$. But we will present it here for the sake of clarity. \\

Let $(P',Q')$ and $(R',T')$ be pair of matrices representing respectively the matrix factorizations $Y\otimes_{S} X$ and $\Delta_{g}$. We assume $g$ is not a constant polynomial, thus $\Delta_{g}$ is not the zero polynomial and so, by corollaries \ref{matrix facto is invertible 1} and \ref{matrix facto is invertible 2}, we have that $R'$ and $T'$ are invertible.\\
 We know that \\
$\rho =(\rho^{0}, \rho^{1}): Y\otimes X \rightarrow 1_{(S,g)}=\Delta_{g}$ \\
is a morphism of matrix factorizations if the following diagram commutes:\\
$\xymatrix{(Y\otimes X)^{0} \ar[r]^{{d^{0}_{Y\otimes X}}} \ar[d]_{\rho^{0}} &
(Y\otimes X)^{1} \ar[d]^{\rho^{1}} \ar[r]^{{d^{1}_{Y\otimes X}}} & (Y\otimes X)^{0}\ar[d]^{\rho^{0}}\\
(\Delta_{g})^{0} \ar[r]^{{d^{0}_{\Delta_{g}}}} & (\Delta_{g})^{1}\ar[r]^{{d^{1}_{\Delta_{g}}}} & (\Delta_{g})^{0}}$\\

That is: $$\begin{cases}
\rho^{0}d^{1}_{Y\otimes X} = d^{1}_{\Delta_{g}}\rho^{1}\\
\rho^{1}d^{0}_{Y\otimes X} = d^{0}_{\Delta_{g}}\rho^{0}
\end{cases}$$


In matrix form: $$\dag\begin{cases}
\rho^{0}Q' = T'\rho^{1}...i)'\\
\rho^{1}P' = R'\rho^{0}...ii)'
\end{cases}$$
where for ease of notation we wrote $\rho^{i}$ for the matrix of $\rho^{i}$, $i=0,1$.\\
Now, $X$ (respectively $Y$) being a matrix factorization of $g(y)-f(x)$ (respectively $f(x)-g(y)$) we obtain from the Yoshino's tensor product of matrix factorization (cf. definition 1.2 of \cite{yoshino1998tensor}) that $Y\otimes X$ is a factorization of $(f(x)-g(y))+(g(y)-f(x))=0$. This means that $P'Q'=0$.

from $ii)'$, since $R'$ is invertible, we have:
\begin{align*}
\rho^{1}P' = R'\rho^{0} &\Rightarrow R'^{-1}\rho^{1}P' =\rho^{0}\\
&\Rightarrow R'^{-1}\rho^{1}P'Q' =\rho^{0}Q'\\
&\Rightarrow \rho^{0}Q' = 0
\end{align*}

Putting this in $i)'$, since $T'$ is invertible, we get:
\begin{align*}
T'\rho^{1}= \rho^{0}Q' = 0 &\Rightarrow \rho^{1}= T'^{-1}0= 0\\
&\Rightarrow \rho^{1}= 0
\end{align*}
%
We now gather our results in the following theorem.
\begin{theorem} \label{neces cond on rho and eta for M.C.in LG}
Let
\begin{itemize}
\item $(R,f)$ and $(S,g)$ be two objects of $\mathcal{LG}_{K}$.
\item $X\in \mathcal{LG}_{K}((R,f),(S,g))= hmf(R\otimes_{K} S, g-f)$ i.e., $X:(R,f) \rightarrow (S,g)$ is a finite rank matrix factorization of $g-f$.
\item $Y\in \mathcal{LG}_{K}((S,g),(R,f))= hmf(S\otimes_{K} R, f-g)$ i.e., $Y:(S,g) \rightarrow (R,f)$ is a finite rank matrix factorization of $f-g$. \\
    such that $X\otimes Y$ and $Y\otimes X$ are finite rank matrix factorizations.
\item $\Delta_{f}\in \mathcal{LG}_{K}((R,f),(R,f))= hmf(R\otimes_{K} R, f\otimes id - id\otimes f)$ i.e., $\Delta_{f}:(R,f) \rightarrow (R,f)$ is a finite rank matrix factorization of $f\otimes id-id\otimes f$.

\item $\Delta_{g}\in \mathcal{LG}_{K}((S,g),(S,g))= hmf(R\otimes_{K} S, g\otimes id - id\otimes g)$ i.e., $\Delta_{g}:(S,g) \rightarrow (S,g)$ is a finite rank matrix factorization of $g\otimes id-id\otimes g$.
\item $\eta: X\otimes_{R} Y \rightarrow 1_{(R,f)}=\Delta_{f}$ and let $(P,Q)$ and $(R,T)$ be pairs of matrices representing respectively the finite rank matrix factorizations $X\otimes Y$ and $\Delta_{f}$.
\item $\rho: Y\otimes_{S} X \rightarrow 1_{(S,g)}=\Delta_{g}$ and let $(P',Q')$ and $(R',T')$ be pairs of matrices representing respectively the finite rank matrix factorizations $Y\otimes X$ and $\Delta_{g}$.
 \end{itemize}
 Then:
 A necessary condition for $\Gamma = (X,Y,\eta,\rho)$ to be a \textit{Morita Context} is $$\begin{cases}
    \eta^{1}=0,\;and\; \eta^{0}Q=0,\\
    \rho^{1}=0,\;and\; \rho^{0}Q'=0.
   \end{cases} $$
\end{theorem}

We will now prove that these necessary conditions are not sufficient thanks to an auxiliary lemma (lemma \ref{auxiliary lemma}) which is proved using a celebrated result (Fact \ref{det block matrices}) on determinant of block matrices due to Schur (e.g \cite{puntanen2005historical}, \cite{silvester2000determinants} (theorem 3)).

\begin{fact} \label{det block matrices}
If $A,B,C$ and $D$ are $n\times n$ matrices, and $$M=\begin{bmatrix}
A & B  \\
C & D
\end{bmatrix},$$
Then the following hold: \begin{enumerate}
\item $det(M)= det(AD-CB)$ if $AC=CA$
\item $det(M)= det(AD-BC)$ if $CD=DC$
\item $det(M)= det(DA-BC)$ if $BD=DB$
\item $det(M)= det(DA-CB)$ if $AB=BA$
\end{enumerate}
\end{fact}
The following auxiliary lemma states that the determinants of the matrices in the matrix factorizations $X\otimes X'$ and $X'\otimes X$ are all equal.
\begin{lemma} \label{auxiliary lemma}
Let $X=(\phi,\psi)$ be an $n\times n$ matrix factorization of $f\in R$ and $X'=(\phi',\psi')$ be an $m\times m$ matrix factorizations of $g\in S$ where $\phi$ and $\psi$ (respectively $\phi'$ and $\psi'$) are matrices over $K[x]$ (respectively $K[y]$). These matrices can be considered as matrices over $L=K[x,y]$ and let\\
\(X\widehat{\otimes} X'=(P,Q)=
(\begin{bmatrix}
    \phi\otimes 1_{m}  &  1_{n}\otimes \phi'      \\
   -1_{n}\otimes \psi'  &  \psi\otimes 1_{m}
\end{bmatrix}
,
\begin{bmatrix}
    \psi\otimes 1_{m}  &  -1_{n}\otimes \phi'      \\
    1_{n}\otimes \psi'  &  \phi\otimes 1_{m}
\end{bmatrix})\)
and \\\\ \(X'\widehat{\otimes} X= (P',Q')=
(\begin{bmatrix}
    \phi'\otimes 1_{n}  &  1_{m}\otimes \phi      \\
   -1_{m}\otimes \psi  &  \psi'\otimes 1_{n}
\end{bmatrix}
,
\begin{bmatrix}
    \psi'\otimes 1_{n}  &  -1_{m}\otimes \phi      \\
    1_{m}\otimes \psi  &  \phi'\otimes 1_{n}
\end{bmatrix})
\)\\

where each component is an endomorphism on $L^{n}\otimes L^{m}$.\\
Then $$\begin{cases}
det(P)=det(P'),\\
det(Q)= det(Q')
\end{cases}$$
Furthermore, all the four determinants are equal.
\end{lemma}
\begin{proof}
We will make use of fact \ref{det block matrices}(2), to compute the determinants of $P$, $P'$, $Q$ and $Q'$ which are block matrices.
Looking at $P$, in order to apply fact \ref{det block matrices}(2), we first have to observe that $(-1_{n}\otimes \psi' )( \psi\otimes 1_{m})=-1_{n}\psi \otimes \psi' 1_{m}=-\psi\otimes \psi'=\psi (-1_{n})\otimes 1_{m}\psi' =(\psi\otimes 1_{m})(-1_{n}\otimes \psi' )$ and looking at $P'$ we equally observe that
$(-1_{m}\otimes \psi )(\psi'\otimes 1_{n})=-1_{m}\psi'\otimes \psi 1_{n}=-\psi'\otimes \psi=\psi' (-1_{m})\otimes 1_{n}\psi =(\psi'\otimes 1_{n})(-1_{m}\otimes \psi )$, we can compute the determinants of $P$ and $P'$ as follows:\\
\begin{align*}
det(P)&= det[ (\phi\otimes 1_{m})(\psi\otimes 1_{m})- (1_{n}\otimes \phi')(-1_{n}\otimes \psi')]\\
&=det[\phi\psi\otimes 1_{m} + 1_{n}\otimes \phi'\psi']\\
&=det[f1_{n}\otimes 1_{m} + 1_{n}\otimes g1_{m}]\,since\,\phi\psi=f1_{n}\,and\,\phi'\psi'=g1_{m}\\
&=det[f(1_{n}\otimes 1_{m}) + g(1_{n}\otimes 1_{m})]\,by\,properties\,of\,\otimes\\
&=det[(f+g)1_{nm}]\,since\;1_{n}\otimes 1_{m}=1_{nm}\\
&=(f+g)^{nm}det(1_{nm})\;by\,properties\;of\;determinants\\
&=(f+g)^{nm}
\end{align*}
\begin{align*}
det(P')&= det[ (\phi'\otimes 1_{n})(\psi'\otimes 1_{n})- (1_{m}\otimes \phi)(-1_{m}\otimes \psi)]\\
&=det[\phi'\psi'\otimes 1_{n} + 1_{m}\otimes \phi\psi]\\
&=det[g1_{m}\otimes 1_{n} + 1_{m}\otimes f1_{n}]\,since\,\phi\psi=f1_{n}\,and\,\phi'\psi'=g1_{m}\\
&=det[g(1_{m}\otimes 1_{n}) + f(1_{m}\otimes 1_{n})]\\
&=det[(g+f)1_{mn}]\,since\;1_{m}\otimes 1_{n}=1_{mn}\\
&=(g+f)^{mn}det(1_{mn})\;by\,properties\;of\;determinants\\
&=(f+g)^{mn}\;by\;commutativity\;in\;K[x,y]\\
&=det(P)\,as\;desired.
\end{align*}
Likewise, in order to use fact \ref{det block matrices}(2) to compute $det(Q)$, from $Q$, we observe that $(1_{n}\otimes \psi' )( \phi\otimes 1_{m})=1_{n}\phi \otimes \psi' 1_{m}=\phi\otimes \psi'=\phi 1_{n}\otimes 1_{m}\psi' =(\phi\otimes 1_{m})(1_{n}\otimes \psi' )$ and looking at $Q'$ we equally observe that
$(1_{m}\otimes \psi )(\phi'\otimes 1_{n})=1_{m}\phi'\otimes \psi 1_{n}=\phi'\otimes \psi=\phi' 1_{m}\otimes 1_{n}\psi =(\phi'\otimes 1_{n})(1_{m}\otimes \psi )$ we can compute the determinants of $Q$ and $Q'$ as follows:\\
\begin{align*}
det(Q)&= det[ (\psi\otimes 1_{m})(\phi\otimes 1_{m})- (-1_{n}\otimes \phi')(1_{n}\otimes \psi')]\\
&=det[\psi\phi\otimes 1_{m} + 1_{n}\otimes \phi'\psi']\\
&=det[f1_{n}\otimes 1_{m} + 1_{n}\otimes g1_{m}]\,since\,\psi\phi=f1_{n}\,and\,\phi'\psi'=g1_{m}\\
&=det[f(1_{n}\otimes 1_{m}) + g(1_{n}\otimes 1_{m})]\,by\,properties\,of\,\otimes\\
&=det[(f+g)1_{nm}]\,since\;1_{n}\otimes 1_{m}=1_{nm}\\
&=(f+g)^{nm}det(1_{nm})\;by\,properties\;of\;determinants\\
&=(f+g)^{nm}
\end{align*}
\begin{align*}
det(Q')&= det[ (\psi'\otimes 1_{n})(\phi'\otimes 1_{n})- (-1_{m}\otimes \phi)(1_{m}\otimes \psi)]\\
&=det[\psi'\phi'\otimes 1_{n} + 1_{m}\otimes \phi\psi]\\
&=det[g1_{m}\otimes 1_{n} + 1_{m}\otimes f1_{n}]\,since\,\psi\phi=f1_{n}\,and\,\phi'\psi'=g1_{m}\\
&=det[g(1_{m}\otimes 1_{n}) + f(1_{m}\otimes 1_{n})]\\
&=det[(g+f)1_{mn}]\,since\;1_{m}\otimes 1_{n}=1_{mn}\\
&=(g+f)^{mn}det(1_{mn})\;by\,properties\;of\;determinants\\
&=(f+g)^{mn}\;by\;commutativity\;in\;K[x,y]\\
&=det(Q)\,as\;desired.
\end{align*}

Clearly, all the four determinants are equal.
\end{proof}

\begin{remark}
It is good to observe that in the case of a Morita Context in $\mathcal{LG}_{K}$, the $f$ and the $g$ we have in the above auxiliary lemma are actually additive inverses of each other by definition of 1-morphisms in $\mathcal{LG}_{K}$. In fact, if $X$ is a morphism from the polynomial $f_{1}$ to the polynomial $f_{2}$, then $X$ is a matrix factorization of $f_{2}-f_{1}=f$. And if $Y$ is a morphism from the polynomial $f_{2}$ to the polynomial $f_{1}$, then $Y$ is a matrix factorization of $f_{1}-f_{2}=g$. We see that $f=-g$.
\end{remark}
Thus, we have the following result which is actually a consequence of lemma \ref{auxiliary lemma}.
\begin{theorem}\label{determinants in MCLGk are all 0}
 Let $(R,f)$ and $(S,g)$ be two objects in $\mathcal{LG}_{K}$ and let $X:(R,f) \rightarrow (S,g)$ and $X': (S,g)\rightarrow (R,f)$ be 1-morphisms in $\mathcal{LG}_{K}$. If $X\otimes X'=(P,Q)$ and $X'\otimes X=(P',Q')$,
 then $det(P)=det(Q)=det(P')=det(Q')=0$
\end{theorem}
\begin{proof}
The proof follows immediately from the remark and lemma \ref{auxiliary lemma}.
\end{proof}
It follows from this theorem that $Q$ in theorem \ref{neces cond on rho and eta for M.C.in LG} is not invertible and so, the equality $\eta^{0}Q=0$ will never have a unique solution for $\eta^{0}$. There would be several solutions from this equality, among which the one(s) that will be sufficient to obtain a \textit{Morita context}.\\
\begin{remark} \label{remark: nec cond not suff}
The fact that $Q$ is not invertible helps to see that
the necessary condition given for $\eta$ in theorem \ref{neces cond on rho and eta for M.C.in LG} is not sufficient. In fact, in the discussion that precedes the statement of that theorem, we cannot reverse the direction of the implication symbol from $\bigstar\bigstar$ to $\bigstar$. Indeed, suppose $\eta^{0}Q=0$ then since $PQ=0$, we have $\eta^{0}Q=R^{-1}\eta^{1}PQ$ implying $R\eta^{0}Q=\eta^{1}PQ$ which implies $(R\eta^{0}-\eta^{1}P)Q=0$. Now, $Q$ being noninvertible, we cannot obtain from here that $(R\eta^{0}-\eta^{1}P)=0$ which is $\bigstar$. So, that necessary condition is not a sufficient one.\\
Similarly, since it also follows from theorem \ref{determinants in MCLGk are all 0} that $Q'$ in theorem \ref{neces cond on rho and eta for M.C.in LG} is not invertible, the necessary condition given for $\rho$ in theorem \ref{neces cond on rho and eta for M.C.in LG} is not sufficient.
\end{remark}

Though the necessary conditions we found are not sufficient, there is a trivial sufficient condition. In fact,
it is easy to see that two equal morphisms of linear factorizations are homotopic; since it would suffice to take $\lambda = 0$ in definition \ref{homotopic lin facto}. We immediately have the following remark which gives a (trivial) sufficient condition to obtain a \textit{Morita context} in $\mathcal{LG}_{K}$.

\begin{remark} \label{suff cond for MC in LG-k}
Let $(R,f)$ and $(S,g)$ be two objects of $\mathcal{LG}_{K}$. Let $X:(R,f) \rightarrow (S,g)$ be a matrix factorization of $g-f$ and $Y:(S,g) \rightarrow (R,f)$ a matrix factorization of $f-g$.
Then
$\Gamma = (X,Y,0,0)$ is a \textit{Morita Context} in $\mathcal{LG}_{K}$. That is, provided we have $X$ and $Y$, it suffices to take $\eta=0$ and $\rho=0$ in definition \ref{defn morita context} to obtain a \textit{Morita Context} in $\mathcal{LG}_{K}$ .
\end{remark}
This follows from the fact that the maps $r$ and $l$ are morphisms of matrix factorizations and so, they are linear. Consequently, the image of zero under these morphisms is zero.\\
In fact:\\
 If $\eta=0$ and $\rho=0$, then:
\begin{enumerate}
\item $\psi=r_{Y}\circ 1_{Y}\otimes\eta \circ a =0$ and
$\phi=l_{Y}\circ \rho\otimes 1_{Y} = 0$
\item $\psi'=r_{X}\circ 1_{X}\otimes\rho \circ a= 0$, and $\phi'=l_{X}\circ \eta \otimes 1_{X}= 0$
\end{enumerate}
Hence, the $M.C.LG_{K}$ diagrams commute up to homotopy if:\\
$\exists \lambda: Z \rightarrow Y$ s.t. $d_{Y}\lambda + \lambda d_{Z}= \psi - \phi= 0$, where $Z=(Y\otimes_{S} X)\otimes_{R} Y$ and $\exists \xi: Z' \rightarrow X$ s.t. $d_{X}\xi + \xi d_{Z'}= \psi' - \phi'= 0$, where $Z'=(X\otimes_{R} Y)\otimes_{S} X$

Hence, it now suffices to choose $\lambda = 0$ and $\xi =0$ to see that $\Gamma = (X,Y,0,0)$ is a \textit{Morita Context} in $\mathcal{LG}_{K}$.\\
A straightforward consequence of theorem \ref{neces cond on rho and eta for M.C.in LG} and remark \ref{suff cond for MC in LG-k} is the following:
\begin{corollary}
A necessary and sufficient condition on $\eta^{1}$ and $\rho^{1}$ for $\Gamma = (X,Y,\eta,\rho)$ to be a \textit{Morita Context} is $\eta^{1}=0=\rho^{1}$.
\end{corollary}

\begin{example}\label{exple MC for suff cond}
Consider $h=-x^{2}+1\in \mathbb{R}[x]=R$ and $g=y^{2}+1\in \mathbb{R}[y]=S$. A \textit{Morita Context} between $(R,h)$ and $(S,g)$ is a quadruple $\Gamma=(X,X',\eta,\rho)$ where:
\begin{itemize}
\item $X:(R,h)\rightarrow (S,g)$ is a matrix factorization of $g-h=y^{2}+x^{2}\in \mathbb{R}[x,y]$. We take \\
\(
X=(\begin{bmatrix}
    x  &  -y      \\
    y  &  x
\end{bmatrix}
,
\begin{bmatrix}
    x  &  y      \\
    -y  &  x
\end{bmatrix})
\)
since \(\begin{bmatrix}
    x  &  -y      \\
    y  &  x
\end{bmatrix}
\begin{bmatrix}
    x  &  y      \\
    -y  &  x
\end{bmatrix}=(g-h)\cdot I_{2}
\)
\item $X':(S,g)\rightarrow(R,h)$ is a matrix factorization of $h-g=-y^{2}-x^{2}\in \mathbb{R}[x,y]$. We take \\
\(
X'=(\begin{bmatrix}
    -x  &  y      \\
    -y  &  -x
\end{bmatrix}
,
\begin{bmatrix}
    x  &  y      \\
    -y  &  x
\end{bmatrix})
\)
since \(\begin{bmatrix}
    -x  &  y      \\
    -y  &  -x
\end{bmatrix}
\begin{bmatrix}
    x  &  y      \\
    -y  &  x
\end{bmatrix}=(h-g)\cdot I_{2}
\)
\item $\eta=X\otimes X'\rightarrow \Delta_{h}$, $x\otimes x' \mapsto 0$ viz. $\eta$ is the zero map.
\item $\rho=X'\otimes X\rightarrow \Delta_{g}$, $x'\otimes x \mapsto 0$ viz. $\rho$ is the zero map.
\end{itemize}
So, $$\Gamma=[(\begin{bmatrix}
    x  &  -y      \\
    y  &  x
\end{bmatrix}
,
\begin{bmatrix}
    x  &  y      \\
    -y  &  x
\end{bmatrix}),(\begin{bmatrix}
    -x  &  y      \\
    -y  &  -x
\end{bmatrix}
,
\begin{bmatrix}
    x  &  y      \\
    -y  &  x
\end{bmatrix}),0,0 ]$$
is a Morita context between $(R,h)$ and $(S,g)$.
\end{example}
\begin{remark}

\begin{enumerate}
\item It is good to mention that in such a setting (remark \ref{suff cond for MC in LG-k}) not all \textit{Morita Contexts} between two objects are identical. In fact, they differ at the level of the matrix factorizations $X$ and $Y$.
\item Intuitively, \textit{Morita contexts} being pre-equivalences, it is not interesting to study cases where the two polynomials $f$ and $g$ are equal.
\end{enumerate}
\end{remark}

\section{FURTHER PROBLEMS}
 In this paper, we gave necessary conditions on $\eta$ and $\rho$ to obtain a \textit{Morita context} in $\mathcal{LG}_{K}$. But the only sufficient condition we were able to give on a $4$-tuple $(X,Y,\eta,\rho)$ to be a \textit{Morita context} was the trivial one, i.e., $\eta=\rho=0$. An interesting question would be to find nontrivial sufficient conditions on $\eta$ and $\rho$.
%
\begin{quote}
  \textbf{Acknowledgments}
\end{quote}
Part of this work was carried out during my Ph.D. studies in mathematics at the University of Ottawa in Canada.
I am grateful to Prof. Dr. Richard Blute who was my Ph.D. supervisor for all the fruitful interactions. \\ I gratefully acknowledge the financial support of the Queen Elizabeth Diamond Jubilee scholarship during my Ph.D. studies.



\bibliography{fomatati_ref}

\begin{thebibliography}{}

\bibitem[Amitsur, 1971]{amitsur1971rings}
Amitsur, S.~A. (1971).
\newblock Rings of quotients and morita contexts.
\newblock {\em Journal of Algebra}, 17(2):273--298.

\bibitem[Anderson and Fuller, 2012]{anderson2012rings}
Anderson, F.~W. and Fuller, K.~R. (2012).
\newblock {\em Rings and categories of modules}, volume~13.
\newblock Springer Science \& Business Media.

\bibitem[Bass, 1962]{bass1962morita}
Bass, H. (1962).
\newblock {\em The Morita theorems}.
\newblock University of Oregon.

\bibitem[Bass and Roy, 1967]{bass1967lectures}
Bass, H. and Roy, A. (1967).
\newblock {\em Lectures on topics in algebraic K-theory}, volume~41.
\newblock Tata Institute of Fundamental Research Bombay.

\bibitem[B{\"o}kstedt and Neeman, 1993]{bokstedt1993homotopy}
B{\"o}kstedt, M. and Neeman, A. (1993).
\newblock Homotopy limits in triangulated categories.
\newblock {\em Compositio Mathematica}, 86(2):209--234.

\bibitem[Camacho, 2015]{camacho2015matrix}
Camacho, A.~R. (2015).
\newblock Matrix factorizations and the landau-ginzburg/conformal field theory
  correspondence.
\newblock {\em arXiv preprint arXiv:1507.06494}.

\bibitem[Carqueville and Murfet, 2015]{carqueville2015toolkit}
Carqueville, N. and Murfet, D. (2015).
\newblock A toolkit for defect computations in landau-ginzburg models.
\newblock In {\em Proc. Symp. Pure Math}, volume~90, page 239.

\bibitem[Carqueville and Murfet, 2016]{carqueville2016adjunctions}
Carqueville, N. and Murfet, D. (2016).
\newblock Adjunctions and defects in landau--ginzburg models.
\newblock {\em Advances in Mathematics}, 289:480--566.

\bibitem[Conrad, 2016]{conrad2016tensor}
Conrad, K. (2016).
\newblock Tensor products.
\newblock {\em Notes of course, available on-line}.

\bibitem[Crisler and Diveris, 2016]{crisler2016matrix}
Crisler, D. and Diveris, K. (2016).
\newblock Matrix factorizations of sums of squares polynomials.
\newblock {\em Diakses pada: http://pages. stolaf.
  edu/diveris/files/2017/01/MFE1. pdf}.

\bibitem[Dummit and Foote, 2004]{dummit2004abstract}
Dummit, D.~S. and Foote, R.~M. (2004).
\newblock {\em Abstract algebra}, volume~3.
\newblock Wiley Hoboken.

\bibitem[Dyckerhoff et~al., 2013]{dyckerhoff2013pushing}
Dyckerhoff, T., Murfet, D., et~al. (2013).
\newblock Pushing forward matrix factorizations.
\newblock {\em Duke Mathematical Journal}, 162(7):1249--1311.

\bibitem[Eisenbud, 1980]{eisenbud1980homological}
Eisenbud, D. (1980).
\newblock Homological algebra on a complete intersection, with an application
  to group representations.
\newblock {\em Transactions of the American Mathematical Society},
  260(1):35--64.

\bibitem[Fomatati, 2019]{fomatati2019multiplicative}
Fomatati, Y.~B. (2019).
\newblock {\em Multiplicative Tensor Product of Matrix Factorizations and Some
  Applications}.
\newblock PhD thesis, Universit{\'e} d'Ottawa/University of Ottawa.

\bibitem[Fomatati, 2021]{fomatati2021tensor}
Fomatati, Y.~B. (2021).
\newblock On tensor products of matrix factorizations.
\newblock {\em arXiv preprint arXiv:2105.10811}.

\bibitem[Goldie, 1958]{goldie1958structure}
Goldie, A.~W. (1958).
\newblock The structure of prime rings under ascending chain conditions.
\newblock {\em Proceedings of the London Mathematical Society}, 3(4):589--608.

\bibitem[Goldie, 1960]{goldie1960semi}
Goldie, A.~W. (1960).
\newblock Semi-prime rings with maximum condition.
\newblock {\em Proceedings of the London Mathematical Society}, 3(1):201--220.

\bibitem[Jacobson, 1964]{jacobson1964structure}
Jacobson, N. (1964).
\newblock Structure of rings, rev. ed.
\newblock In {\em Amer. Math. Soc. Colloq. Publ}, volume~37.

\bibitem[Kaoutit, 2006]{kaoutit2006wide}
Kaoutit, L.~E. (2006).
\newblock Wide morita contexts in bicategories.
\newblock {\em arXiv preprint math/0608601}.

\bibitem[Keller et~al., 2011]{keller2011two}
Keller, B., Murfet, D., and Van~den Bergh, M. (2011).
\newblock On two examples by iyama and yoshino.
\newblock {\em Compositio Mathematica}, 147(2):591--612.

\bibitem[Khovanov and Rozansky, 2008]{khovanov2008matrix}
Khovanov, M. and Rozansky, L. (2008).
\newblock Matrix factorizations and link homology ii.
\newblock {\em Geometry \& Topology}, 12(3):1387--1425.

\bibitem[Lam, 1999]{lam1999graduate}
Lam, T.-Y. (1999).
\newblock Graduate texts in mathematics.

\bibitem[Morita, 1958]{morita1958duality}
Morita, K. (1958).
\newblock Duality for modules and its applications to the theory of rings with
  minimum condition.
\newblock {\em Science Reports of the Tokyo Kyoiku Daigaku, Section A},
  6(150):83--142.

\bibitem[Neeman, 2001]{neeman2001triangulated}
Neeman, A. (2001).
\newblock {\em Triangulated categories}.
\newblock Princeton University Press.

\bibitem[P{\'e}csi, 2012]{pecsi2012morita}
P{\'e}csi, B. (2012).
\newblock On morita contexts in bicategories.
\newblock {\em Applied Categorical Structures}, 20(4):415--432.

\bibitem[Puntanen and Styan, 2005]{puntanen2005historical}
Puntanen, S. and Styan, G.~P. (2005).
\newblock Historical introduction: Issai schur and the early development of the
  schur complement.
\newblock In {\em The Schur complement and its applications}, pages 1--16.
  Springer.

\bibitem[Silvester, 2000]{silvester2000determinants}
Silvester, J.~R. (2000).
\newblock Determinants of block matrices.
\newblock {\em The Mathematical Gazette}, 84(501):460--467.

\bibitem[Smith, 2011]{smith2011introduction}
Smith, R.~A. (2011).
\newblock Introduction to vector spaces, vector algebras, and vector
  geometries.
\newblock {\em arXiv preprint arXiv:1110.3350}.

\bibitem[Yoshino, 1998]{yoshino1998tensor}
Yoshino, Y. (1998).
\newblock Tensor products of matrix factorizations.
\newblock {\em Nagoya Mathematical Journal}, 152:39--56.

\end{thebibliography}
\addcontentsline{toc}{section}
{References}

\end{document}